\documentclass[12pt,a4paper]{article}
\usepackage{amsfonts}
\usepackage{amsmath}
\usepackage{amssymb}
\usepackage{amsthm}
\usepackage[english]{babel}
\usepackage[T1]{fontenc}
\usepackage{graphicx}
\usepackage{subfigure}
\usepackage{epsfig}
\usepackage{latexsym}
\usepackage{amstext}
\usepackage{graphics}
\usepackage{enumerate}
\usepackage{fullpage}
\newtheorem{prop}{Proposition}[section]
\newtheorem{rem}[prop]{Remark}
\newtheorem{lem}[prop]{Lemma}
\newtheorem{theo}[prop]{Theorem}
\newtheorem{cor}[prop]{Corollary}

\newcommand{\beq}{\begin{eqnarray}}
\newcommand{\beqq}{\begin{eqnarray*}}
\newcommand{\eeq}{\end{eqnarray}}
\newcommand{\eeqq}{\end{eqnarray*}}

\title{Bougerol's identity in law and extensions}
\author{ S. Vakeroudis \thanks{Laboratoire de Probabilit\'{e}s et Mod\`{e}les
Al\'{e}atoires (LPMA) CNRS : UMR7599,  Universit\'{e} Pierre et Marie
Curie - Paris VI,  Universit\'{e} Paris-Diderot - Paris VII, 4 Place Jussieu, 75252 Paris Cedex 05, France.
E-mail: stavros.vakeroudis@upmc.fr }
\thanks{Probability and Statistics Group, School of Mathematics, University of Manchester,
Alan Turing Building, Oxford Road, Manchester M13 9PL, United Kingdom. } }
\date{\today}
\begin{document}

\maketitle
\begin{abstract}
We present a list of equivalent expressions and extensions of Bougerol's celebrated identity in law,
obtained by several authors. We recall well-known results and the latest progress of the
research associated with this celebrated identity in many directions, we give some new results
and possible extensions and we try to point out open questions.
\end{abstract}

$\vspace{5pt}$

$\vspace{5pt}$

\tableofcontents

\section{Introduction}\label{secintro}
\renewcommand{\thefootnote}{\fnsymbol{footnote}}
Bougerol's celebrated identity in law has been the subject of research for several authors since first formulated
in 1983 \cite{Bou83}. A reason for this study is on the one hand its interest from the mathematical point of view
and on the other hand its numerous applications, namely in Finance (pricing of Asian options etc.)-see e.g. \cite{Yor92,Duf00,Yor01}.
However, one still feels that some better understanding remains to be discovered.

This paper is essentially an attempt to collect all the known results (up to now) and to give a (full)
survey of the several different equivalent expressions and extensions (to other processes, multidimensional
versions, etc.) in a concise way. We also provide a bibliography, as complete as possible.
For the extended proofs we address the reader to the original articles.

Bougerol's remarkable identity states that
(see e.g. \cite{Bou83,ADY97} and \cite{Yor01} (p. 200)), with
$(B_{u},u\geq0)$ and $(\beta_{u},u\geq0)$ denoting two independent linear Brownian
motions\footnote[4]{When we simply write: Brownian motion, we always mean real-valued Brownian motion starting from 0.
For 2-dimensional Brownian motion we indicate planar or complex BM.}, we have:
\beq\label{boug}
    \mathrm{for} \; \mathrm{fixed} \; t, \ \ \sinh(B_{t}) \stackrel{(law)}{=}
    \beta_{A_{t}(B)} \ ,
\eeq
where $A_{u}(B)=\int^{u}_{0}ds \exp(2B_{s})$  is independent of $(\beta_{u},u\geq0)$.
For a first approach of (\ref{boug}), see e.g. the corresponding Chapters in \cite{ReY99} and in \cite{ChY12}.
In what follows, sometimes for simplicity
we will use the notation $A_{u}$ instead of $A_{u}(\cdot)$.\\
Alili, Dufresne and Yor \cite{ADY97} obtained the following simple proof of Bougerol's identity (\ref{boug}):
\begin{proof}
On the one hand, we define $S_{t}\equiv\sinh(B_{t})$; then, applying It\^{o}'s formula we have:
\beq\label{S}
S_{t}=\int^{t}_{0} \sqrt{1+S_{s}^{2}} \ dB_{s}+\frac{1}{2}\int^{t}_{0}S_{s} \ ds \ .
\eeq
On the other hand, a time-reversal argument for Brownian motion yields: for fixed $t\geq 0$,
\beq\label{Q}
\beta_{A_{t}(B)}=\int^{t}_{0} e^{B_{s}}d\gamma_{s}\stackrel{(law)}{=}
e^{B_{t}}\int^{t}_{0} e^{-B_{s}}d\gamma_{s}\equiv Q_{t} \ ,
\eeq
where $(\gamma_{s},s\geq 0)$ denotes another 1-dimensional Brownian motion, independent
from $(B_{s},s\geq 0)$. \\
Applying once more It\^{o}'s formula to $Q_{t}$, we have:
\beq\label{Q1}
dQ_{t}=\frac{1}{2} Q_{t}dt+(Q_{t}dB_{t}+d\gamma_{t})=\frac{1}{2} Q_{t}dt+\sqrt{Q_{t}^{2}+1} \ d\delta_{t},
\eeq
where $\delta$ is another 1-dimensional Brownian motion, depending on $B$ and on $\gamma$.
From (\ref{S}) and (\ref{Q1}) we deduce that $S$ and $Q$
satisfy the same Stochastic Differential Equation with Lipschitz coefficients, hence, we obtain (\ref{boug}).
\end{proof}
With some elementary computations, from (\ref{boug}) (e.g. identifying the densities of both sides,
for further details see \cite{Vakth11,BDY12a}), we may
obtain the Gauss-Laplace transform of the clock $A_{t}$: for every $x\in\mathbb{R}$,
with $a(x)\equiv\arg\sinh(x)\equiv\log\left(x+\sqrt{1+x^{2}}\right)$
\beq
E\left[\frac{1}{\sqrt{A_{t}}} \exp\left(-\frac{x^{2}}{2A_{t}}\right)\right]=
\frac{a'(x)}{\sqrt{t}} \exp\left(-\frac{a^{2}(x)}{2t}\right).
\eeq
where $a'(x)=(1+x^{2})^{-1/2}$. \\
For further use, we note that Bougerol's identity may be equivalently stated as:
\beq\label{bougerol2abs}
    \sinh(|B_{u}|) \stackrel{(law)}{=} |\beta|_{A_{u}(B)}.
\eeq
Using now the symmetry principle (see \cite{And87} for the original note and \cite{Gal08} for a detailed discussion):
\beq\label{bougerol2sup}
    \sinh(\bar{B}_{u}) \stackrel{(law)}{=} \bar{\beta}_{A_{u}(B)},
\eeq
where, e.g. $\bar{B}_{u}\equiv\sup_{0\leq s\leq u} B_{s}$.

In the remainder of this article we give several versions and generalizations of Bougerol's identity (\ref{boug}).
In particular, in Section \ref{secbougpr} we give extensions of this identity to other processes
(i.e. Brownian motion with drift, hyperbolic Brownian motion, etc.). Section \ref{secsub}
is devoted to some results that we obtain from subordination and some applications to the study
of Bougerol's identity in terms of planar Brownian motion and of  complex-valued Ornstein-Uhlenbeck processes.
In Section \ref{secmulti} we give some 2 and 3 dimensional extensions of Bougerol's identity, first involving
the local time at 0 of the Brownian motion $B$, and second by studying the joint law of 2 and 3 specific processes.
In particular, in Subsection \ref{subsec2dnew} we give a new 2-dimensional extension.
In Section \ref{secbougdif} we generalize Bougerol's identity for the case of diffusions, named \textit{"Bougerol's diffusions"},
followed by some studies in terms of Jacobi processes. Section \ref{secbougpeacock} deals with Bougerol's identity
from the point of view of \textit{"peacocks"} (see this Section for the precise definition, as introduced in e.g. \cite{HPRY11}).
In Section \ref{secopen} we propose some possible directions for further investigation of this "mysterious" identity
in law with its versions and extensions and we give an as full as possible list of references (to the best of author's knowledge) up to now.
Finally, in the Appendix, we present several tables of Bougerol's identity and all the equivalent forms and extensions that we present
in this survey. These tables can be read independently from the rest of the text.

We also note that (sometimes) the notation used from Section to Section may be independent.

\section{Extensions of Bougerol's identity to other processes}\label{secbougpr}

\subsection{Brownian motions with drifts}
Alili, Dufresne and Yor, in \cite{ADY97}, showed the following result:
\begin{prop}\label{propADY}
With $\mu,\nu$ two real numbers, for every $x$ fixed, the Markov process:
\beq\label{Xmunu}
X^{(\mu,\nu)}_{t}\equiv \left(\exp(B_{t}+\mu t)\right)\left(x+\int^{t}_{0} \exp\left(-(B_{s}+\mu s)\right)d(\beta_{s}+\nu s)\right),
\eeq
for every $t\geq0$, has the same law as $(\sinh (Y^{(\mu,\nu)}_{t}), \ t\geq0)$, where $(Y^{(\mu,\nu)}_{t},t\geq0)$
is a diffusion with infinitesimal generator:
\beq
\frac{1}{2} \ \frac{d^{2}}{dy^{2}}+ \left(\mu \tanh(y) + \frac{\nu}{\cosh(y)}\right) \frac{d}{dy} \ ,
\eeq
starting from $y= \arg \sinh (x)$.
\end{prop}
\begin{proof}
It suffices to apply It\^{o}'s formula to both processes
$X^{(\mu,\nu)}$ and $\sinh (Y^{(\mu,\nu)})$.
\end{proof}
It follows now:
\begin{cor}\label{corADY}
For every $t$ fixed,
\beq
\sinh (Y^{(\mu,\nu)}_{t})\stackrel{(law)}{=} \int^{t}_{0} \exp(B_{s}+\mu s)d(\beta_{s}+\nu s).
\eeq
In particular, in the case $\mu=1$ and $\nu=0$:
\beq
\sinh (B_{t}+\varepsilon t)\stackrel{(law)}{=} \int^{t}_{0} \exp(B_{s}+s)d\beta_{s},
\eeq
with $\varepsilon$ denoting a symmetric Bernoulli variable taking values in $\{-1,1\}$.
\end{cor}
\begin{rem}
With $\mu=-1/2$ and $\nu=0$, we have that $\sinh\left(Y^{(-1/2,0)}_{t}\right)$ is a martingale.
Indeed, with $Y_{t}\equiv Y^{(-1/2,0)}_{t}$, It\^{o}'s formula yields:
\beqq
\sinh(Y_{t})&=&\int^{t}_{0} \cosh(Y_{s}) \ dY_{s}+\frac{1}{2}\int^{t}_{0}\sinh(Y_{s}) \ ds \\
&=& \int^{t}_{0} \cosh(Y_{s}) \left[dB_{s}-\frac{1}{2} \tanh(Y_{s}) \ ds \right] + \frac{1}{2}\int^{t}_{0}\sinh(Y_{s}) \ ds \\
&=& \int^{t}_{0} \cosh(Y_{s}) dB_{s} \ .
\eeqq
Hence:
\beq
M_{t}\equiv\sinh(Y_{t})=\beta_{\int^{t}_{0} ds \left(\cosh^{2}(Y_{s})\right)}\equiv
\beta_{\int^{t}_{0} ds \left(1+\sinh^{2}(Y_{s})\right)} \ ,
\eeq
and for this Markovian martingale, we have:
\beq
M_{t}=\sinh(Y_{t})=\int^{t}_{0} \cosh(Y_{s}) dB_{s}=\int^{t}_{0} \sqrt{1+M_{s}^{2}} dB_{s} \ .
\eeq
It can also be seen directly from (\ref{Xmunu}) that $\left(X^{(-1/2,0)}_{t},t\geq 0\right)$ is the product
of two orthogonal martingales.
This property is true because:
\beq
X^{(-1/2,0)}_{t}=\frac{B_{u}}{R_{u}} \Big|_{u=A^{(1/2)}_{t}} \ ,
\eeq
with $A_{t}^{(\nu)}=\int^{t}_{0}ds \ \exp (2B_{s}^{(\nu)})$,  $(B_{t}^{(\nu)},t\geq 0)$
denoting a Brownian motion with drift, and $(R_{t},t\geq 0)$
a 2-dimensional Bessel process started at 0.
Further details about this ratio are discussed in Sections \ref{secbougdif} and \ref{secopen}.
We also remark that, with the notation of Section \ref{secintro}, $A_{t}^{(0)}\equiv A_{t}$.
\end{rem}

\subsection{Hyperbolic Brownian motion}
Alili and Gruet in \cite{AlG97} generalized Bougerol's identity in terms of hyperbolic Brownian motion:
\begin{prop}\label{AG}
We use the notation introduced in the previous Subsection, that is: $(R_{t},t\geq 0)$ is a 2-dimensional Bessel process with $R_{0}=0$
and wee also denote by $\Xi$ an arcsine variable such that $B^{(\nu)}$,
$R$ and $\Xi$ are independent. Let $\phi$ be the function defined by:
\beq
\phi(x,z)=\sqrt{2e^{x}\cosh(z)-e^{2x}-1}, \ \ for \ \ z\geq|x|.
\eeq
Then, for fixed $t$, we have:
\beq
\beta_{A_{t}^{(\nu)}}\stackrel{(law)}{=} (2\Xi-1) \phi\left(B_{t}^{(\nu)},\sqrt{R_{t}^{2}+(B_{t}^{(\nu)})^{2}}\right).
\eeq
In particular, with $\nu=0$, we recover Bougerol's identity:
\beq
\beta_{A_{t}}\stackrel{(law)}{=} (2\Xi-1) \phi\left(B_{t},\sqrt{R_{t}^{2}+B_{t}^{2}}\right)
\stackrel{(law)}{=} \sinh(B_{t}).
\eeq
\end{prop}
This is an immediate consequence of the following:
\begin{lem}\label{lemAG}
\begin{enumerate}[(i)]
 \item  The law of the functional $A_{t}^{(\nu)}$ is characterized by: for all $u\geq0$,
\beq
E\left[\exp\left(-\frac{u^{2}}{2} \ A_{t}^{(\nu)}\right)\right]=e^{-\nu^{2}t/2} \int_{\mathbb{R}} dx \ e^{\nu x}
\int^{+\infty}_{|x|} dz \ \frac{z}{\sqrt{2\pi t^{3}}} \ e^{-z^{2}t/2} J_{0}(u\phi(x,z)),
\eeq
where $J_{0}$ stands for the Bessel function of the first kind with parameter 0 \cite{Leb72}.
\item In particular, taking $\nu=0$, for $u\geq0$ and $x\in\mathbb{R}$ we have:
\beq
\exp\left(-\frac{x^{2}}{2t}\right) E\left[\exp\left(-\frac{u^{2}}{2} \ A_{t}\right) \big| B_{t}=x \right]=
\int^{+\infty}_{|x|} dz \ \frac{z}{t} \ e^{-z^{2}t/2} J_{0}(u\phi(x,z)).
\eeq
\end{enumerate}
\end{lem}
Proposition \ref{AG} follows now immediately from Lemma \ref{lemAG} by using the classical representation of the
Bessel function of the first kind with parameter 0 (see e.g. \cite{Leb72}):
\beq
J_{0}(z)=\frac{1}{\pi} \int^{+1}_{-1} \frac{dr}{\sqrt{1-r^{2}}} \ \cos(zr),
\eeq
and remarking that (with $\Xi$ denoting again an arcsine variable), for all real $\xi$:
\beq
J_{0}(\xi)=E\left[\exp\left(i\xi (2\Xi-1)\right)\right].
\eeq
\begin{proof} \emph{(Lemma \ref{lemAG})} \\
With $I_{\mu}$ and $K_{\mu}$ denoting the modified Bessel functions of the first and the second kind respectively
with parameter $\mu=\sqrt{\rho^{2}+\nu^{2}}$ (for $\rho$ and $\nu$ reals),
we define the function $G_{\mu}:\mathbb{R}^{2}\rightarrow\mathbb{R}^{+}$ by:
\beq
G_{\mu}(u,v)=\left\{
  \begin{array}{ll}
    2I_{\mu}(u)K_{\mu}(v), & u\leq v ; \\
    2I_{\mu}(v)K_{\mu}(u), & u\geq v .
  \end{array}
\right.
\eeq
First, using the skew product representation of planar Brownian motion, the following formula holds
(for further details we address the interested reader to \cite{AlG97}):
\beq\label{invert}
\int^{\infty}_{0}dt \ \exp\left(-\frac{\rho^{2}}{2} \ t\right) E\left[\exp\left(-\frac{u^{2}}{2} \ A_{t}^{(\nu)}\right)\right]
=\int^{+\infty}_{-\infty} dy \ e^{\nu y}G_{\mu}(u,ue^{y}).
\eeq
Using the integral representation (see e.g. \cite{Leb72}, problem 8, p. 140):
\beq
I_{\mu}(x)K_{\mu}(y)=\frac{1}{2} \int^{\infty}_{\log(y/x)} dr \ e^{-\mu r} J_{0}\left(\sqrt{2\cosh(r) xy-x^{2}-y^{2}}\right), \ \ y\geq x.
\eeq
we can invert (\ref{invert}) in order to obtain part $\left.i\right)$ of Lemma \ref{lemAG}. \\
Part $\left.ii\right)$ follows with the help of Cameron-Martin relation.
\end{proof}

\section{Bougerol's identity and subordination}\label{secsub}
In this Section, we consider $(Z_{t}=X_{t}+iY_{t},t\geq0)$ a standard planar Brownian motion (BM)
starting from $x_{0}+i0,x_{0}>0$ (for simplicity and without loss of generality, we suppose that $x_{0}=1$).
Then, a.s., $(Z_{t},t\geq0)$ does not visit $0$ but it winds around it infinitely often, hence
$\theta_{t}=\mathrm{Im}(\int^{t}_{0}\frac{dZ_{s}}{Z_{s}}),t\geq0$
is well defined \cite{ItMK65}. There is the well-known
skew-product representation:
\beq\label{skewproduct}
\log\left|Z_{t}\right|+i\theta_{t}\equiv\int^{t}_{0}\frac{dZ_{s}}{Z_{s}}=\left(
B_{u}+i\gamma_{u}\right)
\Bigm|_{u=H_{t}=\int^{t}_{0}\frac{ds}{\left|Z_{s}\right|^{2}}} \ ,
\eeq
where $(B_{u}+i\gamma_{u},u\geq0)$ is another planar Brownian motion starting from $\log 1+i0$. Thus:
\beqq
H^{-1}_{u}\equiv \inf\{ t:H_{t}>u \} =
\int^{u}_{0}ds \exp(2B_{s}) := A_{u}(B) .
\eeqq
For further study of the Bessel clock $H$, see e.g. \cite{Yor80}.
We also define the first hitting times
$T^{\theta}_{c}\equiv \inf\{ t:\theta_{t}=c \}$ and $T^{|\theta|}_{c}\equiv\inf\{ t:|\theta_{t}|=c \}$.

\subsection{General results}\label{subsecgeneral}
Bougerol's identity in law combined with the symmetry principle of Andr\'{e} \cite{And87,Gal08} yields the following identity in law
(see e.g. \cite{BeY12,BDY12a}): for every fixed $l>0$,
\beq\label{subord}
H_{\tau_{l}}\stackrel{(law)}{=}\tau_{a(l)}
\eeq
where $(\tau_{l},l\geq 0)$ stands for a stable (1/2)-subordinator.
An example of this kind of identities in law is given for the planar Brownian motion case in the next Subsection.
The main point in \cite{BeY12} is that (\ref{subord}) is not extended in the level of processes indexed by $l\geq 0$.

\subsection{Bougerol's identity in terms of planar Brownian motion}
Vakeroudis \cite{Vak11} investigated Bougerol's identity in terms of planar Brownian motion
and obtained some striking identities in law:
\begin{prop}\label{remid} Let $(\beta_{u},u\geq0)$ be a 1-dimensional Brownian
motion independent of the planar Brownian motion $(Z_{u},u\geq0)$
starting from 1. Then, for any $b\geq0$ fixed, the following identities in law hold:
\beqq
    \left.i\right) \ H_{T^{\beta}_{b}} \stackrel{(law)}{=} T^{B}_{a(b)} \ \ \ \
    \left.ii\right) \ \theta_{T^{\beta}_{b}} \stackrel{(law)}{=} C_{a(b)} \ \ \ \
    \left.iii\right) \ \bar{\theta}_{T^{\beta}_{b}} \stackrel{(law)}{=}
    |C_{a(b)}|,
\eeqq
where $C_{A}$ is a Cauchy variable with parameter $A$ and
$\bar{\theta}_{u}=\sup_{0\leq s\leq u}\theta_{s}$.
\end{prop}
\begin{proof}
$\left.i\right)$ We identify the laws of the first hitting times of a fixed
level $b$ by the processes on each side of (\ref{bougerol2sup}) and we obtain:
\beqq
T^{B}_{a(b)} \stackrel{(law)}{=} H_{T^{\beta}_{b}},
\eeqq
which is $\left.i\right)$. \\
$\left.ii\right)$ It follows from $\left.i\right)$ since:
\beqq
\theta_{u} \stackrel{(law)}{=} \gamma_{H_{u}},
\eeqq
with $(\gamma_{s},s\geq0)$ a Brownian motion independent of
$(H_{u},u\geq0)$ and $(C_{u},u\geq0)$ may be represented as
$(\gamma_{T^{B}_{u}},u\geq0)$. \\
$\left.iii\right)$ follows from $\left.ii\right)$ again with the help of the symmetry
principle.
\end{proof}
Using now these identities in law, we can apply William's "pinching" method \cite{Wil74,MeY82}
and recover Spitzer's celebrated asymptotic law which states that \cite{Spi58}:
\beq\label{Spi}
 \frac{2}{\log t} \; \theta_{t} \overset{{(law)}}{\underset{t\rightarrow\infty}\longrightarrow} C_{1} \ ,
\eeq
with $C_{1}$ denoting a standard Cauchy variable
(for other proofs, see also e.g. \cite{Wil74,Dur82,MeY82,BeW94,Yor97,VaY11a}).
One can also find a characterization of the distribution of $T^{\theta}_{c}$ and of $T^{|\theta|}_{c}$ in \cite{Vak11}.
First, applying Bougerol's identity (\ref{boug}) in terms of planar Brownian motion
we have:
\begin{prop}
For fixed $c>0$,
\beq\label{Bougerolplanar}
\sinh(C_{c}) \stackrel{(law)}{=}\beta_{(T^{\theta}_{c})}\stackrel{(law)}{=}\sqrt{T^{\theta}_{c}} \ N \ ,
\eeq
where $N\sim\mathcal{N}(0,1)$ and the involved random variables are independent.
\end{prop}
Furthermore, we can obtain the following Gauss-Laplace transforms
which are equivalent to Bougerol's identity exploited for planar Brownian motion:
\begin{prop}\label{GLtrans}
For $x \geq 0$ and $m=\frac{\pi}{2c}$,
\beq
 c \; E \left[ \sqrt{\frac{\pi}{2 T^{\theta}_{c}}} \exp \left( -\frac{x}{2T^{\theta}_{c}} \right) \right] &=& \frac{1}{\sqrt{1+x}} \: \frac{c^{2}}{(c^{2}+\log^{2}(\sqrt{x}+\sqrt{1+x}))} \ ; \label{GLtheta} \\
 c \; E \left[ \sqrt{\frac{2}{\pi T^{|\theta|}_{c}}} \exp \left( -\frac{x}{2T^{|\theta|}_{c}} \right) \right] &=&
 \frac{1}{\sqrt{1+x}} \frac{2}{(\sqrt{1+x}+\sqrt{x})^{m}+(\sqrt{1+x}-\sqrt{x})^{m}} \ .  \nonumber \\
 \label{GLthetacone}
\eeq
\end{prop}
\begin{proof}
For the proof of (\ref{GLtheta}), it suffices to identify the densities of the two parts of (\ref{Bougerolplanar})
and to recall that the density of a Cauchy variable with parameter $c$ equals:
$$\frac{c}{\pi(c^{2}+y^{2})} \ .$$
For (\ref{GLthetacone}), we apply Bougerol's identity with $u=T_{c}^{|\gamma|} \equiv \inf
\{ t: |\gamma_{t}|=c \}$ and we obtain:
\beq\label{Bougerolplanarcone}
\sinh(B_{T_{c}^{|\gamma|}}) \stackrel{(law)}{=}\beta_{(T^{|\theta|}_{c})}\stackrel{(law)}{=}\sqrt{T^{|\theta|}_{c}} \ N \ .
\eeq
Once again we identify the densities of the two parts. For the left hand side, we use the following Laplace transform:
for $\lambda\geq0$, $E\left[e^{-\frac{\lambda^{2}}{2}T^{|\gamma|}_{b}}\right]=\frac{1}{\cosh(\lambda b)}$
(see e.g. Proposition 3.7, p. 71 in Revuz and Yor \cite{ReY99}).
We also use the well-known result \cite{Lev80,BiY87}:
\beq\label{Fourier}
E \left[ \exp( i\lambda B_{T^{|\gamma|}_{c}}) \right] &=& \frac{1}{\cosh(\lambda c)} = \frac{1}{\cosh(\pi \lambda \frac{c}{\pi})}
= \int^{\infty}_{-\infty} e^{i \left( \frac{\lambda c}{\pi} \right) x  } \frac{1}{2\pi} \frac{1}{\cosh(\frac{x}{2})} \ dx \ . \nonumber \\
\eeq
Changing now the variables $y=cx/\pi$, we obtain the density of $B_{T^{|\gamma|}_{c}}$ which equals: $$\left( 2c \cosh(\frac{y\pi}{2c})\right)^{-1} = \left(c(e^{\frac{y\pi}{2c}} + e^{-\frac{y\pi}{2c}})\right)^{-1},$$
and finishes the proof.
\end{proof}
Vakeroudis and Yor in \cite{VaY11a,VaY11b} investigated further the law of these random times.

\subsection{The Ornstein-Uhlenbeck case}
Vakeroudis in \cite{Vak11,Vak12} investigated also the case of Ornstein-Uhlenbeck processes. In particular,
we consider now a complex valued Ornstein-Uhlenbeck (OU) process:
\beq\label{OUequation}
    Z_{t} = z_{0} + \tilde{Z_{t}} - \lambda \int^{t}_{0} Z_{s} ds,
\eeq
where $\tilde{Z_{t}}$ is a complex valued Brownian motion, $z_{0}\in \mathbb{C}$
(for simplicity and without loss of generality, we suppose again $z_{0}=1$), $\lambda \geq 0$ and
$T^{(\lambda)}_{c} \equiv T^{|\theta^{Z}|}_{c} \equiv \inf
\left\{t\geq 0 : \left|\theta^{Z}_{t}\right|=c \right\}$
($\theta^{Z}_{t}$ is the continuous winding process associated to
$Z$) denoting the first hitting time of the symmetric conic
boundary of angle $c$ for $Z$. Then, we have the following:
\begin{prop}\label{OUbougerol}
Consider $(Z^{\lambda}_{t},t\geq 0)$ and $(U^{\lambda}_{t},t\geq
0)$ two independent Ornstein-Uhlenbeck processes, the first one complex valued and the second one real valued,
both starting from a point different from 0, and define
$T^{(\lambda)}_{b}(U^{\lambda})= \inf \left\{t\geq 0 : e^{\lambda
t } U^{\lambda}_{t}=b \right\}$, for any $b\geq0$. Then, an Ornstein-Uhlenbeck
extension of identity in law $\left.ii\right)$ in Proposition \ref{remid} is the
following:
\beq\label{OUbougerolequation}
    \theta^{Z^{\lambda}}_{T^{(\lambda)}_{b}(U^{\lambda})}  \stackrel{(law)}{=} C_{a(b)},
\eeq
where $a(x)= \arg \sinh (x)$ and $C_{A}$ is a Cauchy variable with parameter $A$.
\end{prop}
\begin{proof}
First, for Ornstein-Uhlenbeck processes, is well known that \cite{ReY99},
with $\left(\mathbb{B}_{t},t\geq0\right)$ denoting a complex valued Brownian motion starting from 1,
Dambis-Dubins-Schwarz Theorem yields:
\beq\label{OUeB}
    Z_{t} &=& e^{-\lambda t} \left( 1 + \int^{t}_{0} e^{\lambda s} d\tilde{Z_{s}} \right) \nonumber \\
          &=& e^{-\lambda t} \left( \mathbb{B}_{\alpha_{t}} \right),
\eeq
Let us consider a second Ornstein-Uhlenbeck process $(U^{\lambda}_{t},t\geq 0)$
independent of the first one. Taking now equation (\ref{OUeB})
for $U^{\lambda}_{t}$ (1-dimensional case) we have:
\beq\label{eU}
    e^{\lambda t } U^{\lambda}_{t}= \delta_{( \frac{e^{2\lambda t}-1}{2 \lambda} )},
\eeq
where $(\delta_{t},t\geq 0)$ is a real valued Brownian
motion starting from 1. \\
Second, applying It\^{o}'s formula to (\ref{OUeB}) and dividing by $Z_{s}$, we obtain
$(\alpha_{t}=\int^{t}_{0} e^{2\lambda s} ds = \frac{e^{2\lambda t}-1}{2 \lambda})$:
\beqq
\mathrm{Im} \left(\frac{dZ_{s}}{Z_{s}}\right) =
\mathrm{Im}
\left(\frac{d\mathbb{B}_{\alpha_{s}}}{\mathbb{B}_{\alpha_{s}}}\right),
\eeqq
hence:
\beq\label{thetaBMOU}
\theta^{Z}_{t} = \theta^{\mathbb{B}}_{\alpha_{t}} \ \ .
\eeq
By inverting $\alpha_{t}$, it follows now that:
\beq\label{Tchat}
    T^{(\lambda)}_{c}=\frac{1}{2\lambda}\ln \left(1+2\lambda T^{|\theta|^{\mathbb{B}}}_{c}\right).
\eeq
Similarly, for the 1-dimensional case we have:
\beq\label{TbU}
    T^{(\lambda)}_{b}(U^{\lambda})=\frac{1}{2\lambda}\ln \left(1+2\lambda T^{\delta}_{b}\right).
\eeq
Equation (\ref{thetaBMOU}) for $t=\frac{1}{2\lambda} \ln
\left(1+2\lambda T^{\delta}_{b}\right)$, equivalently:
$\alpha(t)=T^{\delta}_{b}$ becomes:
\beqq
\theta^{Z^{\lambda}}_{T^{(\lambda)}_{b}(U^{\lambda}) } =
\theta^{Z^{\lambda}}_{\frac{1}{2\lambda}\ln \left(1+2\lambda
T^{\delta}_{b}\right) }= \theta^{\mathbb{B}}_{u=T^{\delta}_{b}}
\stackrel{(law)}{=} C_{a(b)},
\eeqq
where the last equation in law follows precisely from statement $\left.ii\right)$ in Proposition \ref{remid}.
\end{proof}

\section{Multidimensional extensions of Bougerol's identity}\label{secmulti}

\subsection{The law of the couple $\left(\sinh(\beta_{t}),\sinh(L_{t})\right)$}
A first 2-dimensional extension of Bougerol's identity was obtained by Bertoin, Dufresne and Yor in \cite{BDY12a}
(for a first draft, see also \cite{DuY11}).
With $(L_{t},t\geq 0)$ denoting the local time at 0 of $B$, we have:
\begin{theo}\label{BDY}
For fixed $t$, the 3 following 2-dimensional random variables are equal in law:
\beq\label{eqBDY}
(\sinh(B_{t}),\sinh(L_{t}))\stackrel{(law)}{=}(\beta_{A_{t}},\exp(-B_{t}) \ \lambda_{A_{t}})
\stackrel{(law)}{=} (\exp(-B_{t}) \ \beta_{A_{t}}, \lambda_{A_{t}}),
\eeq
where $(\lambda_{u},u\geq 0)$ is the local time of $\beta$ at 0.
\end{theo}
\begin{rem}\label{BDYrem}
Theorem \ref{BDY} can be equivalently stated as: for fixed $t$, the 3 following  2-dimensional random variables are equal in law:
\beq\label{eqBDY2}
(\sinh(|B_{t}|),\sinh(L_{t}))\stackrel{(law)}{=}(|\beta|_{A_{t}},\exp(-B_{t}) \ \lambda_{A_{t}})
\stackrel{(law)}{=} (\exp(-B_{t}) \ |\beta|_{A_{t}}, \lambda_{A_{t}}).
\eeq
Using now Paul L\'{e}vy's celebrated identity in law (see e.g. \cite{ReY99}):
\beq
\left((\bar{B}_{t}-B_{t},\bar{B}_{t}),t\geq0\right)\stackrel{(law)}{=}\left((|B_{t}|,L_{t}),t\geq0\right),
\eeq
we can reformulate (\ref{eqBDY}) or (\ref{eqBDY2}), and we obtain:
\beq\label{eqBDY3}
(\sinh(\bar{B}_{t}-B_{t}),\sinh(\bar{B}_{t}))&\stackrel{(law)}{=}&\left((\bar{\beta}-\beta)_{A_{t}},\exp(-B_{t}) \ \bar{\beta}_{A_{t}}\right) \nonumber \\
&\stackrel{(law)}{=}& \left(\exp(-B_{t}) \ (\bar{\beta}-\beta)_{A_{t}}, \bar{\beta}_{A_{t}}\right).
\eeq
The latter is particularly interesting when compared with the Wiener-Hopf factorization for Brownian motion.
In particular, if we consider $\mathbf{e}_{q}$ an independent exponential random variable of parameter $q$,
then $\bar{B}_{\mathbf{e}_{q}}$ is independent of $B_{\mathbf{e}_{q}}-\bar{B}_{\mathbf{e}_{q}}$. This tells that
the two random variables appearing on the right hand side of (\ref{eqBDY3}), when taken at $\mathbf{e}_{q}$, are independent.
\end{rem}
\begin{rem}\label{BDYrem2}
Considering only the second processes of the first and the third part of (\ref{eqBDY}) (or equivalently of
(\ref{eqBDY2})), we obtain a "local time" version of Bougerol's identity:
\beq\label{eqBDYlocaltimes}
\sinh(L_{t})\stackrel{(law)}{=}\lambda_{A_{t}},
\eeq
which (as was shown in \cite{BeY12}), similar to the Brownian motion case, is true only
for fixed $t$ and not in the level of processes.
\end{rem}
\begin{proof} \emph{(Theorem \ref{BDY})} \\
From Remark \ref{BDYrem} it suffices to prove (\ref{eqBDY2}). \\
First, we denote $S_{p}$, $p\geq0$ an exponential variable with parameter $p$
independent from $B$ and $g_{t}=\sup\{u<t:B_{u}=0\}$. We know that
$\left(B_{u}, u\leq g_{S_{p}}\right)$ and $\left(B_{g_{S_{p}}+u}, u\leq S_{p}-g_{S_{p}}\right)$ are independent,
hence $L_{S_{p}}$ and $B_{S_{p}}$ are also independent.
We also know that $L_{t}$ and $|B_{t}|$ have the same law. Hence, using the following computation:
for every $l\geq0$, with $(\tau_{l},l\geq 0)$ denoting the time $L$ reaches $l$,
$$P\left(L_{S_{p}}\geq l\right)=P\left(S_{p}\geq \tau_{l}\right)=E\left[\exp(-p \tau_{l})\right]=\exp(-l \sqrt{2p}),$$
we deduce that the common density of $L_{S_{p}}$ and $|B_{S_{p}}|$ is:
$$\sqrt{2p}\exp(-u \sqrt{2p}), \ \ u\geq 0.$$
Equivalently, we have:
$$\sqrt{2\mathbf{e}}(|\beta(1)|,\lambda(1))\stackrel{(law)}{=}(\mathbf{e},\mathbf{e'}),$$
where on the left hand side $\mathbf{e}$ and $\mathbf{e'}$ are two independent copies of $S_{1}$
independent from $\beta$. \\
For the second identity in law in Theorem \ref{BDY}, it suffices to remark that
$$(\beta_{A_{t}},\exp(-B_{t}) \ \lambda_{A_{t}})
\stackrel{(law)}{=} (\sqrt{A_{t}}\beta_{1},\exp(-B_{t}) \sqrt{A_{t}} \ \lambda_{1}),$$
and use a time reversal argument. \\
For the first identity in law we use an exponential time $S_{p}$ and we
compute the joint Mellin transforms in both sides in order to
show that:
$$\sqrt{2\mathbf{e}}(\sinh(|B|_{S_{p}}),\sinh(L_{S_{p}}))\stackrel{(law)}{=}
\sqrt{2\mathbf{e}}(\exp(-B_{S_{p}})\sqrt{A_{S_{p}}} \ |\beta_{1}|, \sqrt{A_{S_{p}}} \ \lambda_{A_{t}}).$$
For further details we address the reader to \cite{BDY12a}.
\end{proof}
Using now Tanaka's formula we can also obtain the following identity in law for 2-dimensional processes:
\begin{cor}\label{BDYcor}
\beq\label{eqBDYcor}
\left(\sinh(B_{t}),L_{t}\right)_{t\geq0}\stackrel{(law)}{=}
\left(\exp(-B_{t}) \ \beta_{A_{t}}, \int^{t}_{0} \exp(-B_{s})d\lambda_{A_{s}}\right)_{t\geq0},
\eeq
where, in each part, the second process is the local time at level 0 and time $t$ of the first one.
\end{cor}

\subsection{Another two-dimensional extension}\label{subsec2dnew}
In this Subsection we will study the joint distribution of:
\beq
	\left( X^{(1)}_{u},X^{(2)}_{u} \right) = \left( \exp(- B_{u}) \int^{u}_{0} d\xi^{(1)}_{v} \exp( B_{v}), \  \exp(- 2B_{u}) \int^{u}_{0} d\xi^{(2)}_{v} \exp( 2B_{v}) \right),
\eeq
where $(\xi^{(1)}_{v},v\geq 0)$, $(\xi^{(2)}_{v},v\geq 0)$ and $(B_{u},u\geq 0)$ are three independent Brownian motions.
Hence, we obtain a new 2-dimensional extension which states the following:
\begin{prop}\label{propVYboug}
We consider $(B^{(1)}_{t}, t\geq0)$ and $(B^{(2)}_{t}, t\geq0)$ two real dependent Brownian motions, such that:
\beq\label{B1B2dep}
 d<B^{(1)},B^{(2)}>_{v} = \tanh(B^{(1)}_{v}) \; \tanh(2B^{(2)}_{v}) \; dv.
\eeq
For the two-dimensional process $\left( X^{(1)}_{u},X^{(2)}_{u} \right)$, we have:
\begin{enumerate}[(i)]
\item In the level of processes:
\beq\label{X1X2process}
\left( X^{(1)}_{u},X^{(2)}_{u}, u\geq0 \right) &\stackrel{(law)}{=}& \left( \sinh (B^{(1)}_{u}), \frac{1}{2} \sinh (2B^{(2)}_{u}), u\geq0  \right)
\eeq
\item For $u$ fixed,
\beq\label{X1X2fixed}
\left( X^{(1)}_{u},X^{(2)}_{u} \right)
&\stackrel{(law)}{=}& \left( \beta^{(1)}_{\left( \int^{u}_{0} dv \: \exp(2B_{v}) \right)} , \beta^{(2)}_{\left( \int^{u}_{0} dv \: \exp(4B_{v}) \right)} \right).
\eeq
\end{enumerate}
\end{prop}
\begin{proof}
Let us define:
\beq
X^{(\alpha)}_{u} =  \exp(- \alpha B_{u}) \int^{u}_{0} d\xi^{(\alpha)}_{v} \exp( \alpha B_{v}),
\eeq
where $\alpha = 1,2$.
By It\^{o}'s formula, we have:
\beqq
	X^{(\alpha)}_{u} &=& \xi^{(\alpha)}_{u} + \int^{u}_{0} \left( \exp(- \alpha B_{v}) (-\alpha dB_{v}) + \frac{\alpha^{2}}{2} \exp(- \alpha B_{v}) dv \right) \; \left( \int^{v}_{0} d\xi^{(\alpha)}_{h} \exp( \alpha B_{h}) \right) \\
	&=& \xi^{(\alpha)}_{u} + \int^{u}_{0} \left(  -\alpha dB_{v} \; X^{(\alpha)}_{v} + \frac{\alpha^{2}}{2} \; X^{(\alpha)}_{v} \; dv \right).
\eeqq
Hence:
\beq
X^{(1)}_{u} &=& \xi^{(1)}_{u} - \int^{u}_{0} dB_{v} \; X^{(1)}_{v} + \frac{1}{2} \int^{u}_{0} X^{(1)}_{v} \; dv \nonumber \\
&=& \int^{u}_{0} d\eta^{(1)}_{v} \sqrt{\left(1+ \left(X^{(1)}_{v}\right)^{2}\right)} + \frac{1}{2} \int^{u}_{0} X^{(1)}_{v} \; dv \ ,
\eeq
and
\beq
X^{(2)}_{u} &=& \xi^{(2)}_{u} - 2 \int^{u}_{0} dB_{v} \; X^{(2)}_{v} + 2 \int^{u}_{0} X^{(2)}_{v} \; dv \nonumber \\
&=& \int^{u}_{0} d\eta^{(2)}_{v} \sqrt{\left(1+ 4\left(X^{(2)}_{v}\right)^{2}\right)} + 2 \int^{u}_{0} X^{(2)}_{v} \; dv,
\eeq
where $(\eta^{(1)}_{v},v\geq 0)$ and $(\eta^{(2)}_{v},v\geq 0)$ are two dependent Brownian motions, with quadratic variation:
\beq
d<\eta^{(1)},\eta^{(2)}>_{v} = \frac{2 X^{(1)}_{v} X^{(2)}_{v} dv}{\sqrt{\left(1+ \left(X^{(1)}_{v}\right)^{2}\right)} \sqrt{\left(1+ 4\left(X^{(2)}_{v}\right)^{2}\right)}}.
\eeq
Thus, we deduce that the infinitesimal generator of $\left( X^{(1)}_{u},X^{(2)}_{u} \right)$ is:
\beq
\frac{1}{2} \left[ \left( 1+x^{2}_{1} \right)\frac{\partial^{2}}{\partial x^{2}_{1}} + \left( 1+4x^{2}_{2} \right)\frac{\partial^{2}}{\partial x^{2}_{2}} + 4 x_{1} x_{2} \frac{\partial^{2}}{\partial x_{1} \partial x_{2}} \right]
+ \frac{x_{1}}{2} \frac{\partial}{\partial x_{1}} + 2 x_{2} \frac{\partial}{\partial x_{2}} \ .
\eeq
Let us now study the couple:
\beq
(x^{(1)}_{t},x^{(2)}_{t}) = \left( \sinh (B^{(1)}_{t}), \frac{1}{2} \sinh (2B^{(2)}_{t}) \right),
\eeq
where $(B^{(1)}_{t}, t\geq 0)$ and $(B^{(2)}_{t}, t\geq 0)$ are two dependent Brownian motions. By It\^{o}'s formula we have:
\beq
	x^{(1)}_{t} &=& \sinh (B^{(1)}_{t}) \nonumber \\
	&=& \int^{t}_{0} \cosh(B^{(1)}_{v}) \; dB^{(1)}_{v} + \frac{1}{2} \int^{t}_{0} \sinh(B^{(1)}_{v}) \; dv \nonumber \\
	&=& \int^{t}_{0} \sqrt{\left( 1+ (x^{(1)}_{v})^{2}\right)} \; dB^{(1)}_{v} + \frac{1}{2} \int^{t}_{0} x^{(1)}_{v} \; dv,
\eeq
and:
\beq
	x^{(2)}_{t} &=& \frac{1}{2} \sinh (2B^{(2)}_{t}) \nonumber \\
	&=& \int^{t}_{0} \cosh(2B^{(2)}_{v}) \; dB^{(2)}_{v} + \int^{t}_{0} \sinh(2B^{(2)}_{v}) \; dv \nonumber \\
	&=& \int^{t}_{0} \sqrt{\left( 1+ 4 (x^{(2)}_{v})^{2}\right)} \; dB^{(2)}_{v} + 2 \int^{t}_{0} x^{(2)}_{v} \; dv.
\eeq
Moreover, using (\ref{B1B2dep}):
\beq
d<\sinh (B^{(1)}),\frac{1}{2} \sinh (2B^{(2)})>_{v}
&=& \cosh(B^{(1)}_{v}) \; \cosh(2B^{(2)}_{v}) \; d<B^{(1)},B^{(2)}>_{v} \nonumber \\
&=& 2 \sinh(B^{(1)}_{v}) \; \frac{1}{2} \sinh(2B^{(2)}_{v}) \; dv.
\eeq
Finally, we have that $(x^{(1)}_{u},x^{(2)}_{u})$ has the same infinitesimal generator with $\left( X^{(1)}_{u},X^{(2)}_{u} \right)$. Hence,
we get part $\left(i\right)$ of the Proposition. \\ \\
For part $\left(ii\right)$, we fix $u$ and we have:
\beq\label{eqlaw}
 \left( \sinh (B^{(1)}_{u}), \frac{1}{2} \sinh (2B^{(2)}_{u}) \right) \stackrel{(law)}{=} \left( \beta^{(1)}_{\left( \int^{u}_{0} dv \: \exp(2B_{v}) \right)} , \beta^{(2)}_{\left( \int^{u}_{0} dv \: \exp(4B_{v}) \right)} \right),
\eeq
where $(\beta^{(1)}_{v},v\geq 0)$ and $(\beta^{(2)}_{v},v\geq 0)$ are two dependent Brownian motions and $(B_{v},v\geq 0)$ is another Brownian motion independent from them. Now, from (\ref{eqlaw}), we obtain (\ref{densityp}).
\end{proof}
\begin{rem}
From (\ref{eqlaw}), with $p_{u}(x,y)$ denoting now the density function of the couple $(\sinh (B^{(1)}_{u}), \frac{1}{2} \sinh (2B^{(2)}_{u}))$, we have:
\beq\label{densityp}
p_{u}(x,y) = E \left[ \frac{1}{2\pi} \frac{\exp \left( -x^{2}/ 2 \int^{u}_{0} dv \: \exp(2B_{v}) \right)}{\sqrt{\int^{u}_{0} dv \: \exp(2B_{v})}} \frac{\exp \left( -y^{2}/ 2 \int^{u}_{0} dv \: \exp(4B_{v}) \right)}{\sqrt{\int^{u}_{0} dv \: \exp(4B_{v})}} \right].
\eeq
In theory, we should be able to compute this probability density as we know the joint distribution of the couple of exponential functionals
(see e.g. \cite{AMSh01}).
\end{rem}

\subsection{A three-dimensional extension}
Alili, Dufresne and Yor, in \cite{ADY97}, obtained a 3-dimensional extension of Bougerol's identity:
\begin{prop}\label{propADY2}
The two following processes have the same law:
\beq
\left\{e^{B_{t}}\int^{t}_{0}e^{-B_{u}}d\beta_{u},B_{t},\beta_{t};t\geq 0\right\}
\stackrel{(law)}{=}\left\{\sinh(B_{t}),B'_{t},G'_{t};t\geq 0\right\},
\eeq
where:
\beq
\left\{
  \begin{array}{ll}
    B'_{t}= \int^{t}_{0} \tanh(B_{s}) dB_{s}+ \int^{t}_{0} \frac{dG_{s}}{\cosh(B_{s})} \ ; \\
    G'_{t}= \int^{t}_{0}\frac{dB_{s}}{\cosh(B_{s})}- \int^{t}_{0}\tanh(B_{s})dG_{s} \ ,
  \end{array}
\right.
\eeq
with $(G_{t},t\geq 0)$ denoting another Brownian motion, independent from $B$.
\end{prop}
\begin{rem}\label{remADY}
We remark that with:
\beq
\alpha(x)=
\begin{pmatrix}
\tanh(x)	&	-\frac{1}{\cosh(x)}\\
\frac{1}{\cosh(x)}	&	\tanh(x)
\end{pmatrix} ,
\eeq
we have:
\beq\label{inv}
\begin{pmatrix}
dB'_{t}\\
dG'_{t}
\end{pmatrix} = \alpha(B_{t})
\begin{pmatrix}
dB_{t}\\
d\beta_{t}
\end{pmatrix} ,
\eeq
and
$$\left\{\begin{pmatrix}
B'_{t}\\
G'_{t}
\end{pmatrix}, t\geq 0\right\}$$
is a 2-dimensional Brownian motion.
\end{rem}
\begin{proof} \emph{(Proposition \ref{propADY2})} \\
\textit{First proof:}
Using It\^{o}'s formula, we deduce easily that each of these triplets is a Markov process with infinitesimal
generator (in $C^{2}(\mathbb{R}^{3})$):
\beq
\frac{1}{2}(1+x^{2})\frac{d^{2}}{dx^{2}}+\frac{1}{2}\frac{d^{2}}{dy^{2}}+\frac{1}{2}\frac{d^{2}}{dz^{2}}
+x\frac{d^{2}}{dxdy}+\frac{d^{2}}{dxdz}+x\frac{d}{dx}.
\eeq
The proof finishes by the uniqueness (in law) of the solutions of the corresponding
martingale problem. \\ \\
\textit{Second proof:} First, we admit that the identity in law is true.
Then, if we replace on the left hand side $(B_{s})$ by $(B'_{s})$ and $(\beta_{s})$ by $(G'_{s})$,
we have necessarily:
\beq\label{secondpr}
\sinh(B_{t})\stackrel{(law)}{=}e^{B'_{t}}\int^{t}_{0}e^{-B'_{u}}dG'_{t},
\eeq
which is essentially a (partial) inversion formula of the transformation (\ref{inv}). \\
Equation (\ref{secondpr}) can be proved by using It\^{o}'s formula on the right hand side.
\end{proof}
Gruet in \cite{ADY97} also remarked that:
\begin{prop}\label{Gruet}
There exist two independent linear Brownian motions $V$ and $W$ and a diffusion $J$
starting from 0 satisfying the following equation
\beq
dJ_{t}=dW_{t}+\frac{1}{2} \ \tanh(J_{t})dt,
\eeq
such that,
\beq
\begin{pmatrix}
d\beta_{t}\\
dB_{t}
\end{pmatrix} = \alpha(-J_{t})
\begin{pmatrix}
dV_{t}\\
dW_{t}
\end{pmatrix} .
\eeq
Hence, the two following 3-dimensional processes:
$$\left(\exp\left(B_{t}+\frac{t}{2}\right) \int^{t}_{0}\exp\left(-B_{s}-\frac{s}{2}\right) d\beta_{s},B_{t},\beta_{t};t\geq 0\right)$$
and $$\left(\sinh(J_{t}),B_{t},\beta_{t};t\geq 0\right),$$
are equal.
\end{prop}
\begin{proof}
This result follows from a geometric proof and it is essentially an explanation of the second
proof of Proposition \ref{propADY2}, at least for $\nu=0$.
For this purpose, we can compare the writing of a hyperbolic Brownian motion
in the half-plane of Poincar\'{e}, decomposed in rectangular coordinates
with the equidistant coordinates \cite{Vin93}. For further details, see the Appendix in \cite{ADY97} due to Gruet.
\end{proof}

\section{The diffusion version of Bougerol's identity}\label{secbougdif}

\subsection{Bougerol's diffusion}
Bertoin, Dufresne and Yor in a recent work \cite{BDY12b} generalized Bougerol's identity in terms of diffusions.
First, we remark that from Proposition \ref{propADY} we have that (see also \cite{ADY97}):
\beq\label{bougADY}
\left(\sinh(B_{t}),t\geq0\right)\stackrel{(law)}{=}\left(\exp(-B_{t})\beta_{A_{t}^{(0)}},t\geq0\right).
\eeq
In particular, using Lamperti's relation (see e.g. \cite{Lam72} or \cite{ReY99})
we can invoke a Bessel process $R^{(\delta)}$ independent from $B$
in order to replace the right hand side of (\ref{bougADY}) by:
$$\left(\exp\left(-B^{(\nu)}_{t}\right)R^{(\delta)}_{A^{(\nu)}_{t}},t\geq0\right),$$
which turns out to be a diffusion (named \textit{Bougerol's diffusion}) with a certain infinitesimal generator.
Hence, we obtain the following:
\begin{theo}\label{bougdif}
With $Z=Z^{(\delta)}$ and $Z'=Z^{(\delta')}$ denoting two independent squared Bessel processes of dimension $\delta=2(1+\mu)$
and $\delta'=2(1+\nu)$ respectively, starting from $z$ and $z'$, the process:
\beq\label{X}
X_{t}\equiv X^{(\nu, \delta)}_{t}\equiv \exp\left(-2B^{(\nu)}_{t}\right)Z_{A^{(\nu)}_{t}}=\frac{Z_{u}}{Z'_{u}}\Bigg|_{u=A^{(\nu)}_{t}}, \ t\geq0,
\eeq
is a diffusion with infinitesimal generator:
\beq\label{generator}
2x(1+x)D^{2}+\left(\delta+(4-\delta')x\right)D,
\eeq
where $\delta'=2(1+\nu)$.
\end{theo}
\begin{rem}
There is a discussion in \cite{JaW12} concerning the particular case where the diffusion with generator given in
(\ref{generator}) is the hyperbolic sine of the radial part of a hyperbolic Brownian motion
(or equivalently the hyperbolic sine of a hyperbolic Bessel process)
of index $\alpha\in(-1/2,\infty)$ (see \cite{JaW12} Theorem 2.25, formula (46), p.15).
In that case, with $R_{t}$ denoting this hyperbolic Bessel process starting from $x$
and $Y_{t}=e^{B_{t}-(\alpha+1/2)t}$, for any $w\geq0$, $t\geq0$,
\beq
\left(\sinh(R_{t}), \ t\geq0\right)\stackrel{(law)}{=}\left(Y_{t}^{-1}S_{\int^{t}_{0}Y_{u}^{2}du}, \ t\geq0\right),
\eeq
where $S$ is a Bessel process of dimension $2(1+\alpha)$ independent of $B$, and $S_{0}=\sinh(x)$.
\end{rem}
\begin{proof} \emph{(Proposition \ref{bougdif})} \\
Applying It\^{o}'s formula to the process $X$, we obtain:
\beq\label{form}
X_{t}=\int^{t}_{0} \exp\left(-2B^{(\nu)}_{u}\right)d\left(Z_{A^{(\nu)}_{u}}\right)
+\int^{t}_{0} Z_{A^{(\nu)}_{u}}d\left(\exp\left(-2B^{(\nu)}_{u}\right)\right).
\eeq
For the second integral in (\ref{form}), It\^{o}'s formula once more yields:
\beqq
d\left(\exp\left(-2B^{(\nu)}_{u}\right)\right)&=&-2\exp\left(-2B^{(\nu)}_{u}\right)\left(dB_{u}+\nu du\right)
+2\exp\left(-2B^{(\nu)}_{u}\right)du \nonumber \\
&=&-2\int^{t}_{0}X_{u}dB_{u}+\underbrace{2(1-\nu)}_{4-\delta'} \int^{t}_{0}X_{u}du \ .
\eeqq
Thus:
\beq\label{form1}
\int^{t}_{0} Z_{A^{(\nu)}_{u}}d\left(\exp\left(-2B^{(\nu)}_{u}\right)\right)
&=&-2\int^{t}_{0}X_{u}dB_{u}+\underbrace{2(1-\nu)}_{4-\delta'} \int^{t}_{0}X_{u}du \ .
\eeq
For the first integral in (\ref{form}), we recall that, with $\gamma$ denoting another Brownian motion
independent from $B$ (thus independent also from $Z$):
\beq
dZ_{s}=2\sqrt{Z_{s}} d\gamma_{s}+\delta \ ds \ .
\eeq
Hence:
\beq\label{form2}
dZ_{A^{(\nu)}_{u}}&=&2\sqrt{Z_{A^{(\nu)}_{u}}} d\gamma_{A^{(\nu)}_{u}}+\delta \exp\left(2B^{(\nu)}_{u}\right) du \nonumber \\
&=& 2\sqrt{Z_{A^{(\nu)}_{u}}} \exp\left(B^{(\nu)}_{u}\right) d\hat{\gamma}_{u}+\delta \exp\left(2B^{(\nu)}_{u}\right) du \ ,
\eeq
with $\hat{\gamma}$ denoting another Brownian motion, depending on $\gamma$ and on $B$. \\
The proof finishes by some elementary computations from (\ref{form}), using (\ref{form1}) and (\ref{form2}).\\
Finally, using Lamperti's relation, which states that:
\beq\label{Lamperti}
\exp\left(2B^{(\nu)}_{t}\right)=Z'_{A^{(\nu)}_{t}} \ ,
\eeq
we obtain the last identity in (\ref{X}).
\end{proof}
We may continue a little further in order to obtain the following result relating
the diffusion $X$ with its reciprocal (recall that: $A^{(\nu)}_{u}=\int^{u}_{0} ds \ \exp\left(2B^{(\nu)}_{s}\right)$):
\begin{cor}\label{difX}
The following relation holds:
\beq
\frac{1}{X^{(\nu, \mu)}_{t}}=X^{(\mu, \nu)}_{\int^{t}_{0}\frac{du}{X^{(\nu, \mu)}_{A^{(\nu)}_{u}}}} \ .
\eeq
\end{cor}
\begin{proof}
It follows easily by some relations involving the changes of time:
\beq
A^{(\nu)}_{t}=\int^{t}_{0} ds \ \exp\left(2B^{(\nu)}_{s}\right) ; \ \ A^{(\mu)}_{t}=\int^{t}_{0} ds \ \exp\left(2B^{(\mu)}_{s}\right) ;
\eeq
\beq
H^{(\nu)}_{u}=\int^{u}_{0}\frac{ds}{Z'_{s}}; \ \ H^{(\mu)}_{u}=\int^{u}_{0}\frac{ds}{Z_{s}} \ .
\eeq
Moreover, we remark that $(H^{(\nu)}_{t})$ is the inverse of $(A^{(\nu)}_{t})$ and
$(H^{(\mu)}_{t})$ is the inverse of $(A^{(\mu)}_{t})$. \\
We also need to use:
\beq
H^{(\nu,\mu)}_{u}=\int^{u}_{0}\frac{ds}{X^{(\nu,\mu)}_{s}}; \ \ H^{(\mu,\nu)}_{u}=\int^{u}_{0}\frac{ds}{X^{(\mu,\nu)}_{s}} \ .
\eeq
Simple calculations now yield:
\beq
H^{(\nu,\mu)}_{t}=H^{(\mu)}_{A^{(\nu)}_{t}} \ ,
\eeq
Finally, using (\ref{X}) we have:
\beq
H^{(\mu)}_{A^{(\nu)}_{t}}=\int^{t}_{0} \frac{ds}{X^{(\nu,\mu)}_{A^{(\nu)}_{s}}} \ ,
\eeq
and we obtain easily the result.
\end{proof}

\subsection{Relations involving Jacobi processes}
In this Subsection, we present a particular study of Theorem \ref{bougdif} in terms of
the Jacobi processes $Y^{(\delta,\delta')}$ as introduced in Warren and Yor \cite{WaY98}
(see also the references therein for Jacobi processes),
due to Bertoin, Dufresne and Yor \cite{BDY12b}. First, we recall some results involving Jacobi processes:
\begin{prop}\emph{(Warren and Yor \cite{WaY98}, Proposition 8)} \\
With $T\equiv\inf\{u:Z_{u}+Z'_{u}=0\}$, there exists a diffusion process $\left(Y_{u}\equiv Y^{\delta, \delta'}_{u},u\geq0\right)$
on $[0,1]$, independent from $Z+Z'$ such that:
\beq
\frac{Z_{u}}{Z_{u}+Z'_{u}}=Y_{\int^{u}_{0}\frac{ds}{Z_{s}+Z'_{s}}}, \ u<T.
\eeq
We remark that $Y'=1-Y$ is the Jacobi process with dimensions $(\delta',\delta)$, and
$Y$ has infinitesimal generator:
\beq
2y(1-y)D^{2}+\left(\delta-(\delta+\delta')y\right)D \ .
\eeq
\end{prop}
Now, $X$ defined in (\ref{X}) and $Y$ can be related as following:
\begin{prop}\label{XY}
The following relation holds:
\beq\label{XY1}
\frac{Y_{w}}{1-Y_{w}}=X_{\int^{w}_{0}\frac{dv}{Y'_{v}}}=X_{\int^{w}_{0}\frac{dv}{1-Y_{v}}} \ ,
\eeq
or equivalently:
\beq\label{XY2}
X_{k}=\frac{Y_{w}}{1-Y_{w}}\Bigg|_{w=\int^{k}_{0}\frac{dv}{1+X_{v}}} \ .
\eeq
\end{prop}
\begin{proof} \emph{(Proposition \ref{XY})} \\
First, from (\ref{X}), we have:
$$X_{t}=\frac{Z_{u}}{Z'_{u}}\Bigg|_{u=A^{(\nu)}_{t}} \ . $$
Conversely,
\beq\label{pr1}
\frac{Z_{u}}{Z'_{u}}=X_{H^{(\nu)}_{u}} \ ,
\eeq
where $H^{(\nu)}_{u}=\int^{u}_{0}\frac{ds}{Z'_{s}}$ is the inverse of $A^{(\nu)}$.
However, using the Jacobi process $Y$,
\beq\label{pr2}
\frac{Z_{u}}{Z'_{u}}=\frac{Y_{w}}{1-Y_{w}} \big|_{w=\mathcal{H}_{u}=\int^{u}_{0}\frac{ds}{Z_{s}+Z'_{s}}} \ ,
\eeq
and moreover:
\beq\label{pr3}
H^{(\nu)}_{u}=\int^{u}_{0}\frac{ds}{Z'_{s}}=\int^{u}_{0} \frac{ds}{(Z_{s}+Z'_{s})} \ \frac{1}{(1-Y_{\mathcal{H}_{s}})}
=\int^{\mathcal{H}_{u}}_{0}\frac{dv}{Y'_{v}} \ .
\eeq
Plugging now (\ref{pr3}) to (\ref{pr1}) and comparing to (\ref{pr2}), we obtain (\ref{XY1}). For (\ref{XY2}), it suffices to remark
that $k\rightarrow\int^{k}_{0}\frac{dv}{1+X_{v}}$ is the inverse of the increasing process
$w\rightarrow\int^{w}_{0}\frac{dv}{1-Y_{v}}$ \ .
\end{proof}

\section{Bougerol's identity and peacocks}\label{secbougpeacock}
Hirsch, Profeta, Roynette and Yor in \cite{HPRY11}, studied the processes which are increasing in the convex order, named \textit{peacocks}
(coming from the French term: Processus Croissant pour l'Ordre Convexe, which yields the acronym PCOC).
Let us first introduce a notation: for $W$ and $V$ two real-valued random variables, $W$ is said to be dominated by $V$
for the convex order if, for every convex function $\psi: \mathbb{R}\rightarrow\mathbb{R}$ such that $E[|\psi(W)|] < \infty$
and $E[|\psi(V)|]<\infty$, we have:
\beq
E[\psi(W)] \leq E[\psi(V)],
\eeq
and we write: $W\stackrel{(c)}{\leq}V$. \\
A process $(G_{t},t\geq0)$ is a peacock if, for every $s\leq t$,
$G_{s}\stackrel{(c)}{\leq} G_{t}$. Kelleler's Theorem now (see e.g. \cite{Kel72,HPRY11,HiR12}) states that,
to every peacock, we can associate a martingale (defined possibly on another probability space than $G$). In other words,
there exists a martingale $(M_{t},t\geq0)$ such that, for every fixed $t\geq0$,
\beq
G_{t}\stackrel{(law)}{=}M_{t} \ .
\eeq
The main subject of \cite{HPRY11} is to give several examples of peacocks and the associated martingales. \\
We return now to Bougerol's identity and we remark that (see also \cite{HPRY11}, paragraph 7.5.4, p. 322),
for every $\lambda\geq0$, $\left(\sinh(\lambda B_{t}), t\geq0\right)$ is a peacock with associated martingale
$\left(\lambda \int^{t}_{0} e^{\lambda \beta_{s}} d\gamma_{s}, t\geq0\right)$ (see e.g. (\ref{Q})). \\
Moreover, for every $\lambda$ real, $\left(e^{-\frac{\lambda^{2}}{2}t}\sinh(\lambda B_{t}), \geq0\right)$
is obviously a peacock, as it is a martingale. This is generalized in the following:
\begin{prop}\label{propHPRY}
\emph{(\cite{HPRY11}, Proposition 7.2)} \\
The process $\left(e^{\mu t}\sinh(\lambda B_{t}), \geq0\right)$ is a peacock if and only if $\mu\geq-\frac{\lambda^{2}}{2}$.
\end{prop}
\begin{proof}
$\left.i\right)$ First, we suppose $\mu\geq-\frac{\lambda^{2}}{2}$.
Then, for $s<t$:
\beqq
e^{\mu t}\sinh(\lambda B_{t})&=&e^{(\mu+\frac{\lambda^{2}}{2} )t}\left(\sinh(\lambda B_{t})e^{-\frac{\lambda^{2}}{2}t}\right)
\stackrel{(c)}{\geq} e^{(\mu+\frac{\lambda^{2}}{2} )s}\left(\sinh(\lambda B_{t})e^{-\frac{\lambda^{2}}{2}t}\right) \\
&\stackrel{(c)}{\geq}& e^{(\mu+\frac{\lambda^{2}}{2} )s}\left(\sinh(\lambda B_{s})e^{-\frac{\lambda^{2}}{2}s}\right).
\eeqq
$\left.ii\right)$ Conversely, It\^{o}-Tanaka's formula yields:
\beqq
E\left[|\sinh(\lambda B_{t})|\right]=E\left[\sinh(\lambda |B_{t}|)\right]=
e^{\frac{\lambda^{2}}{2}t}\lambda \int^{t}_{0} \frac{ds}{\sqrt{2\pi s}} \ e^{-\frac{\lambda^{2}}{2}s},
\eeqq
hence:
\beqq
E\left[|e^{\mu t} \sinh(\mu B_{t})|\right] \stackrel{t\rightarrow+\infty}{\thicksim}
\lambda e^{(\frac{\lambda^{2}}{2}+\mu)t}\lambda \int^{+\infty}_{0} \frac{ds}{\sqrt{2\pi s}} \ e^{-\frac{\lambda^{2}}{2}s},
\eeqq
which means that if $\mu<-\frac{\lambda^{2}}{2}$,
\beqq
E\left[|e^{\mu t}\sinh(\mu B_{t})|\right] \stackrel{t\rightarrow+\infty}{\longrightarrow}0 \ .
\eeqq
However, $x\rightarrow|x|$ is convex and if $\left(e^{\mu t}\sinh(\mu B_{t}),t\geq0\right)$ was a peacock,
then $E\left[|e^{\mu t}\sinh(\mu B_{t})|\right]$ would increase on $t$, which is a contradiction.
\end{proof}

\section{Further extensions and open questions}\label{secopen}
In this Section, we propose some possible directions to continue studying
and possibly extending Bougerol's celebrated identity in law (for fixed time or as a process).

First, the natural question posed is wether this identity can be extended to higher dimensions.
This very challenging question has already been attempted to be dealt with, and in this paper
we've presented several extensions, at least for the 2-dimensional (and partly for the 3-dimensional) case.

Another natural question is wether we can generalize Bougerol's identity to other processes.
For this purpose, we may think in terms of a diffusion, as introduced in Section \ref{secbougdif}.
It seems more intelligent to start from the right hand side of (\ref{boug}) and try to see, e.g. in (\ref{X}), for every
particular ratio of processes, which is the corresponding process on the left hand side
(this process could be named "\textit{Bougerol's process}").

In particular, it seems interesting to investigate a possible extension in the case of L\'{e}vy or stable processes.
To that end, we could replace the ratio of the two squared independent Bessel processes in (\ref{X})
by e.g. the ratio of two exponentials of L\'{e}vy processes, and investigate the process obtained after the time-change.
However, this perspective is not in the aims of the present work.

Finally, another aspect which could be further studied is the applications that one may obtain by the
subordination method, as presented in Section \ref{secsub}. Following the lines of this Section,
one may retrieve further results and applications, others than for the planar Brownian motion case (see also \cite{BeY12,BDY12a}).

\appendix

\section{Appendix: Tables of Bougerol's Identity and other equivalent expressions}\label{Bougeroltable}
\vspace{10pt}
Using now the notations introduced in the whole text, we can summarize all the results
in the following tables ($u>0$, wherever used is considered as fixed).
\vspace{10pt}

\subsection{Table: Bougerol's Identity in law and equivalent expressions ($u>0$ fixed)}
With $a(x)\equiv \arg \sinh (x)$, and $B$, $\beta$ denoting two independent real Brownian motions,
for $u>0$ fixed, we have: \\
\begin{center}
\begin{tabular}{|l|c|}
  \hline \\
  $\left.1\right)$ \ $\sinh(B_{u}) \stackrel{(law)}{=} \beta_{(A_{u}(B)\equiv\int^{u}_{0}ds\exp(2B_{s}))}$ \ \ \ \ \ (Bougerol's Identity) \\ \\ \hline \\
  $\left.2\right)$ \ $\sinh(|B_{u}|) \stackrel{(law)}{=} |\beta|_{(A_{u}(B))}$ \\ \\ \hline \\
  $\left.3\right)$ \ $\sinh(\bar{B}_{u}) \stackrel{(law)}{=} \bar{\beta}_{(A_{u}(B))}$, \ \ $\bar{B}_{u}=\sup_{s\leq u}\beta_{s}$ \\ \\ \hline \\
  $\left.4\right)$ \ $E \left[ \frac{1}{\sqrt{2\pi A_{u}(B)}} \exp \left( -\frac{x}{2A_{u}(B)} \right) \right] = \frac{1}{\sqrt{2\pi u}} \: \frac{1}{\sqrt{1+x}} \: \exp \left( -\frac{(a(\sqrt{x}))^{2}}{2u} \right)$, \ $x \geq 0$ \\ \\ \hline
\end{tabular}
\end{center}
\vspace{20pt}

\subsection{Table: Bougerol's Identity for other 1-dimensional processes ($u>0$ fixed)}
We use $\mu$, $\nu$ reals and we define: $$B^{(\mu)}_{t}=B_{t}+\mu t, \ \beta^{(\nu)}_{s}=\beta_{t}+\nu t,$$  $$A_{t}^{(\nu)}=\int^{t}_{0}ds \ \exp (2B_{s}^{(\nu)}),$$ $\varepsilon$: a Bernoulli variable in $\{-1,1\}$, $(R_{t},t\geq 0)$
a 2-dimensional Bessel process started at 0, $\Xi$ an arcsine variable, and $(Y^{(\mu,\nu)}_{t},t\geq0)$ a diffusion with infinitesimal generator:
$$\frac{1}{2} \ \frac{d^{2}}{dy^{2}}+ \left(\mu \tanh(y) + \frac{\nu}{\cosh(y)}\right) \frac{d}{dy} \ , $$
starting from $y= \arg \sinh (x)$. $B^{(\mu)}$, $\beta^{(\nu)}$, $\varepsilon$, $\Xi$ and $R$ are independent. Then, for $u>0$ fixed:
\begin{center}
\begin{tabular}{|l|c|}
  \hline
  \\
  $\left.5\right)$ \ $\left(\sinh(Y^{(\mu,\nu)}_{t}),t\geq0\right) \stackrel{(law)}{=}\left(\exp(B^{(\mu)}_{t})\left(x+\int^{t}_{0}\exp(-B^{(\mu)}_{s})d\beta^{(\nu)}_{s}\right),t\geq0\right)$, $x$: fixed \\ \\
  \hline \\
  $\left.6\right)$ \ $\sinh(Y^{(\mu,\nu)}_{u}) \stackrel{(law)}{=}\int^{u}_{0}\exp(B^{(\mu)}_{s})d\beta^{(\nu)}_{s}$, \\ \\ \hline \\
  $\left.7\right)$ \ $\sinh(B_{u}+\varepsilon \ t) \stackrel{(law)}{=}\int^{u}_{0}\exp(B_{s}+s)d\beta_{s}$, \ $\varepsilon$: Bernoulli variable in $\{-1,1\}$ \\ \\ \hline \\
  $\left.8\right)$ \ $\beta_{A_{u}^{(\nu)}}\stackrel{(law)}{=} (2\Xi-1) \phi\left(B_{u}^{(\nu)},\sqrt{R_{u}^{2}+(B_{u}^{(\nu)})^{2}}\right)$,
  \ $\phi(x,z)=\sqrt{2e^{x}\cosh(z)-e^{2x}-1}$, $z\geq|x|$ \\ \\ \hline
\end{tabular}
\end{center}
\vspace{20pt}

\subsection{Table: Bougerol's Identity in terms of planar Brownian motion ($u>0$ fixed)}
We define $(Z_{t},t\geq0)$ a planar Brownian motion starting from 1.
Then $\theta_{t}=\mathrm{Im}(\int^{t}_{0}\frac{dZ_{s}}{Z_{s}}),t\geq0$
is well defined. We further define the Bessel clock $H_{t}=\int^{t}_{0}\frac{ds}{\left|Z_{s}\right|^{2}}=A^{-1}_{u}(B)$
and the first hitting times:
$T^{\theta}_{c}\equiv \inf\{t:\theta_{t}=c \}$ and $T^{|\theta|}_{c}\equiv\inf\{t:|\theta_{t}|=c \}$.
Then, with $(C_{c},c\geq0)$ a standard Cauchy process, $C_{y}$ a Cauchy variable with parameter $y$, $N\sim\mathcal{N}(0,1)$,
$(Z^{\lambda}_{t},t\geq 0)$ and $(U^{\lambda}_{t},t\geq 0)$
two independent Ornstein-Uhlenbeck processes, the first one complex valued and the second one real valued,
both starting from a point different from 0 and
$T^{(\lambda)}_{b}(U^{\lambda})= \inf \left\{t\geq 0 : e^{\lambda
t } U^{\lambda}_{t}=b \right\}$, for $b,c>0$ fixed:
\begin{center}
\begin{tabular}{|l|c|}
  \hline \\
  $\left.9\right)$ \ $\sinh(C_{c}) \stackrel{(law)}{=}\beta_{(T^{\theta}_{c})}\stackrel{(law)}{=}\sqrt{T^{\theta}_{c}}N$, $c>0$ fixed \\ \\ \hline \\
  $\left.10\right)$ \ $H_{T^{\beta}_{b}} \stackrel{(law)}{=} T^{B}_{a(b)}$, \ \ \ $T^{B}_{y}=\inf\{t:B_{t}=y\}$, $a(x)= \arg \sinh (x)$, \ \ $b>0$ fixed \\ \\ \hline \\
  $\left.11\right)$ \ $\theta_{T^{\beta}_{b}} \stackrel{(law)}{=} C_{a(b)}$, $b>0$ fixed \\ \\ \hline \\
  $\left.12\right)$ \ $\bar{\theta}_{T^{\beta}_{b}} \stackrel{(law)}{=} |C_{a(b)}|$, \ \ \ $\bar{\theta}_{u}=\sup_{s\leq u}\theta_{s}$ \\ \\ \hline \\
  $\left.13\right)$ \ $E \left[ \frac{1}{\sqrt{2\pi T^{\theta}_{c}}} \exp \left( -\frac{x}{2T^{\theta}_{c}} \right) \right] = \frac{1}{\sqrt{1+x}} \: \frac{c}{\pi(c^{2}+\log^{2}(\sqrt{x}+\sqrt{1+x}))}$, \ \ \ $b>0$ fixed, $x \geq 0$ \\ \\ \hline \\
  $\left.14\right)$ \ $E \left[ \frac{1}{\sqrt{2\pi T^{|\theta|}_{c}}} \exp (-\frac{x}{2T^{|\theta|}_{c}} ) \right] = \left( \frac{1}{c} \right) \left(\frac{1}{\sqrt{1+x}} \right) \frac{1}{(\sqrt{1+x}+\sqrt{x})^{\zeta}+(\sqrt{1+x}-\sqrt{x})^{\zeta}}$, $x \geq 0$, $\zeta=\frac{\pi}{2c}$ \\ \\ \hline \\
  $\left.15\right)$ \ $\theta^{Z^{\lambda}}_{T^{(\lambda)}_{b}(U^{\lambda})}  \stackrel{(law)}{=} C_{a(b)}$ \ \ \ (OU version) \\ \\
  \hline
\end{tabular}
\end{center}
\vspace{20pt}

\subsection{Table: Multi-dimensional extensions of Bougerol's Identity}
In the following table, $(L_{t},t\geq 0)$ and $(\lambda_{t},t\geq 0)$ denote the local times at 0 of $B$, $\beta$ respectively and:
$$ \left( X^{(1)}_{u},X^{(2)}_{u} \right) = \left( \exp(- B_{u}) \int^{u}_{0} d\xi^{(1)}_{v} \exp( B_{v}), \  \exp(- 2B_{u}) \int^{u}_{0} d\xi^{(2)}_{v} \exp( 2B_{v}) \right),$$
where $(\xi^{(1)}_{v},v\geq 0)$, $(\xi^{(2)}_{v},v\geq 0)$ and $(B_{u},u\geq 0)$ are three independent Brownian motions.
Moreover, we denote by $(B^{(1)},B^{(2)})$ and $(\beta^{(1)},\beta^{(2)})$ two couples of dependent Brownian motions
(independent from $B$), such that:
$$d<B^{(1)},B^{(2)}>_{v} = \tanh(B^{(1)}_{v}) \; \tanh(2B^{(2)}_{v}) \; dv,$$
and, for $u>0$ fixed:
$$\left\{
  \begin{array}{ll}
  \sinh (B^{(1)}_{u})\stackrel{(law)}{=} \beta^{(1)}_{\left( \int^{u}_{0} dv \: \exp(2B_{v}) \right)} \ ; \\
  \frac{1}{2}\sinh (2B^{(2)}_{u})\stackrel{(law)}{=}\beta^{(2)}_{\left( \int^{u}_{0} dv \: \exp(4B_{v}) \right)} \ .
  \end{array}
\right.
$$
Finally,
$$\left\{
  \begin{array}{ll}
    B'_{t}= \int^{t}_{0} \tanh(B_{s}) dB_{s}+ \int^{t}_{0} \frac{dG_{s}}{\cosh(B_{s})} \ ; \\
    G'_{t}= \int^{t}_{0}\frac{dB_{s}}{\cosh(B_{s})}- \int^{t}_{0}\tanh(B_{s})dG_{s} \ ,
  \end{array}
\right.
$$
with $(G_{t},t\geq 0)$ denoting another Brownian motion, independent from $B$ and
$J$ a diffusion starting from 0 satisfying: $dJ_{t}=dW_{t}+\frac{1}{2} \ \tanh(J_{t})dt$, where $W$ stands for
an independent Brownian motion. Hence, for $u>0$ fixed:
\begin{center}
\begin{tabular}{|l|c|}
  \hline \\
  $\left.16\right)$ \ $\sinh(L_{u})\stackrel{(law)}{=}\lambda_{A_{u}}$ \\ \\ \hline \\
  $\left.17\right)$ \ $(\sinh(B_{u}),\sinh(L_{u}))\stackrel{(law)}{=}(\beta_{A_{u}},\exp(-B_{u}) \ \lambda_{A_{u}})
  \stackrel{(law)}{=} (\exp(-B_{u}) \ \beta_{A_{u}}, \lambda_{A_{u}})$ \\ \\ \hline \\
  $\left.18\right)$  $(\sinh(|B_{u}|) \ \sinh(L_{u}))\stackrel{(law)}{=}(|\beta|_{A_{u}},\exp(-B_{u}) \ \lambda_{A_{u}})
  \stackrel{(law)}{=} (\exp(-B_{u}) \ |\beta|_{A_{u}}, \lambda_{A_{u}})$ \\ \\ \hline \\
  $\left.19\right)$ \ $(\sinh(\bar{B}_{u}-B_{u}),\sinh(\bar{B}_{u})) \stackrel{(law)}{=}\left((\bar{\beta}-\beta)_{A_{u}},\exp(-B_{u}) \ \bar{\beta}_{A_{u}}\right) \stackrel{(law)}{=} (\exp(-B_{u}) \ (\bar{\beta}-\beta)_{A_{u}}, \bar{\beta}_{A_{u}})$ \\ \\ \hline \\
  $\left.20\right)$ \ $\left(\sinh(B_{t}),L_{t}, \ t\geq0\right)\stackrel{(law)}{=}
  \left(\exp(-B_{t}) \ \beta_{A_{t}}, \int^{t}_{0} \exp(-B_{s})d\lambda_{A_{s}}, \ t\geq0\right)$ \\ \\ \hline \\
  $\left.21\right)$ \
  $\left( X^{(1)}_{t},X^{(2)}_{t}, t\geq0 \right)\stackrel{(law)}{=}\left( \sinh (B^{(1)}_{t}), \frac{1}{2} \sinh (2B^{(2)}_{t}), t\geq0 \right)$
  \\ \\ \hline \\
  $\left.22\right)$ \
  $\left( X^{(1)}_{u},X^{(2)}_{u} \right)\stackrel{(law)}{=} \left( \beta^{(1)}_{\left( \int^{u}_{0} dv \: \exp(2B_{v}) \right)} , \beta^{(2)}_{\left( \int^{u}_{0} dv \: \exp(4B_{v}) \right)} \right)$ \\ \\ \hline \\
  $\left.23\right)$ \ $\left(e^{B_{t}}\int^{t}_{0}e^{B_{u}}d\beta_{u},B_{t},\beta_{t};t\geq 0\right)
  \stackrel{(law)}{=}\left(\sinh(B_{t}),B'_{t},G'_{t};t\geq 0\right)$ \\ \\ \hline \\
  $\left.24\right)$ \ $\left(\exp\left(B_{t}+\frac{t}{2}\right) \int^{t}_{0}\exp\left(-B_{s}-\frac{s}{2}\right) d\beta_{s},B_{t},\beta_{t};t\geq 0\right)\stackrel{(law)}{=} \left(\sinh(J_{t}),B_{t},\beta_{t};t\geq 0\right)$ \\ \\ \hline
\end{tabular}
\end{center}
\vspace{20pt}

\subsection{Table: Diffusion version of Bougerol's Identity (relations involving the Jacobi process)}
Let $Z\equiv Z^{(\delta)}$ and $Z'\equiv Z^{(\delta')}$ be two independent squared Bessel process of dimension $\delta=2(1+\mu)$
and $\delta'=2(1+\nu)$ respectively, starting from $z$ and $z'$, and $X_{t}\equiv X^{(\nu, \delta)}_{t}$
a diffusion (named "Bougerol's diffusion"), with infinitesimal generator:
$$2x(1+x)D^{2}+\left(\delta+(4-\delta')x\right)D,$$
and $Y\equiv Y^{\delta,\delta'}$ the Jacobi process. Then, for $t,w,k>0$:
\begin{center}
\begin{tabular}{|l|c|}
  \hline \\
  $\left.25\right)$ \ $X^{(\nu, \delta)}_{t}\equiv \exp\left(-2B^{(\nu)}_{t}\right)Z_{A^{(\nu)}_{t}}=\frac{Z_{u}}{Z'_{u}}\Big|_{u=A^{(\nu)}_{t}}$ \\ \\ \hline \\
  $\left.26\right)$ \ $\frac{Y_{w}}{1-Y_{w}}=X_{\int^{w}_{0}\frac{dv}{Y'_{v}}}=X_{\int^{w}_{0}\frac{dv}{1-Y_{v}}}$ \\ \\ \hline \\
  $\left.27\right)$ \ $X_{k}=\frac{Y_{w}}{1-Y_{w}}\Big|_{w=\int^{k}_{0}\frac{dv}{1+X_{v}}}$ \\ \\ \hline
\end{tabular}
\end{center}

\vspace{30pt}
\noindent\textbf{Acknowledgements} \\
The author is indebted to Professor Marc Yor for useful comments and
advice, and to Professor Ron Doney for his invitation as a Post Doc fellow
at the University of Manchester, where he prepared this survey.
The author would also like to thank an anonymous referee for providing
useful comments and for pointing out the last part of Remark \ref{BDYrem}
concerning the comparison with the Wiener-Hopf factorization of Brownian motion
and references \cite{AMSh01} and \cite{JaW12}.


\end{document}